\documentclass{article}

\usepackage[all,pdf,2cell]{xy}\UseAllTwocells\SilentMatrices
\usepackage{amsthm}
\usepackage{amsmath}
\usepackage{amsfonts}
\usepackage{amssymb}
\usepackage{enumerate}
\usepackage{cleveref}
\usepackage{graphicx}
\newtheorem{theorem}{Theorem}[section]
\newtheorem{lemma}[theorem]{Lemma}
\newtheorem{corollary}[theorem]{Corollary}

\theoremstyle{definition}
\newtheorem{definition}[theorem]{Definition}
\newtheorem{example}[theorem]{Example}
\newtheorem{problem}{Problem}
\newtheoremstyle{sltheorem}
{2pt}                
{2pt}                
{\slshape}        
{}                
{\slshape}       
{:}               
{ }               
{ }                
\theoremstyle{sltheorem}
\newtheorem{claim}{Claim}
\newcommand{\newcategory}[1]{\expandafter\newcommand\csname #1\endcsname{\mathbf{#1}}}
\newcommand{\pp}{{\perp\perp}}

\newcommand{\cl}{\mathrm{cl}}
\newcommand{\pow}{\mathcal{P}}

\begin{document}
\title{N-free posets and orthomodularity}
\author{Gejza Jenča}
\maketitle
\begin{abstract}
We prove that the incomparability orthoset of
a finite poset is Dacey if and only if the poset is N-free. We give a characterization of finite posets with compatible incomparability orthosets.
\end{abstract}
\section{Introduction}

In his PhD. thesis \cite{dacey1968orthomodular}, Dacey explored the notion of
``abstract orthogonality'', by means of sets equipped with a symmetric,
irreflexive relation $\perp$. He named these structures {\em orthogonality spaces}.
Recently, they have been renamed {\em orthosets}
\cite{paseka2022normal,paseka2022categories} and we will use this terminology
in the present paper. 

Every orthoset has an orthocomplementation operator $X\mapsto X^\perp$
defined on the set of all its subsets. Dacey proved that $X\mapsto X^\pp$ is a
closure operator and that the set of all closed subsets of an orthoset forms
a complete ortholattice, which we call {\em the logic of an orthoset}.
Moreover, Dacey gave a characterization of orthosets such that their logic is
an orthomodular lattice. The orthosets of this type are nowadays called {\em Dacey
spaces} or {\em Dacey graphs} \cite{sumner1974dacey}.

Let us remark that (as a special case of a {\em polarity between two sets}) the
idea of an orthoset and its associated lattice of closed sets
appears already in the classical monograph \cite[Section
V.7]{birkhoff1948lattice}.  Nevertheless, Dacey was probably the first person
that explored the orthomodular law in this context.

Since orthosets are nothing but (undirected, simple, loopless)
graphs, it is perhaps not surprising that the $X\mapsto X^\perp$ mapping appears
under the name {\em neighbourhood operator} in graph theory. Implicitly, the logic
associated with the neighbourhood operator was used in the seminal paper
\cite{lovasz1978kneser}, in which Lovász proved Kneser's conjecture.
Explicitly, they were used by Walker in \cite{walker1983graphs}, where 
the ideas from \cite{lovasz1978kneser} were generalized and reformulated in the language of category theory.

For an overview of the results on orthosets and more general
{\em test spaces}, see \cite{wilce2011test}. See also \cite{paseka2022normal,paseka2022categories} for some more recent results
on orthosets.

In \cite{jenca2023orthogonality} we have associated to every finite bounded poset $P$ 
an orthoset $(Q^+(P),\perp)$ consisting of non-singleton closed intervals. We have proved that a finite bounded poset is a lattice if and only if the associated orthoset is Dacey.

In the present paper, we continue this line of research. To every finite poset
$P$, we associate an orthoset $(P,\perp)$, where $\perp$ denotes
incomparability of elements. We call this orthoset {\em the incomparability
orthoset of $P$}. We prove that the incomparability orthoset of a finite poset is
Dacey if and only if the poset is N-free in the sense of \cite{habib1985n}.  In
addition, we characterize orthosets that give rise to a Boolean algebra as
their logic, we call them {\em compatible orthosets}. We then use this result
to give a structural characterization of finite posets with compatible
incomparability orthosets.

\section{Preliminaries}

\subsection{Posets, ortholattices}

We assume working knowledge of posets and lattices. For here undefined notions and notations, see for example \cite{stanley1986enumerative,davey2002introduction}.

Let $P$ be a poset. For $x,y\in P$, we denote $x\prec y$ 
and say that {\em $y$ covers x} or that {\em
$x$ is covered by $y$} if $x<y$ and there is no $z\in P$ such that $x<z<y$.

An {\em ortholattice} is a bounded lattice $(L,\vee,\wedge,0,1,^\perp)$
equipped with an antitone mapping called {\em orthocomplementation}
$\perp\colon L\to L$ such that
\begin{itemize}
\item $0^\perp=1$, $1^\perp=0$
\item $x^\pp=x$
\item $(x\vee y)^\perp=x^\perp\wedge y^\perp$
\item $(x\wedge y)^\perp=x^\perp\vee y^\perp$
\item $x\wedge x^\perp=0$
\item $x\vee x^\perp=1$
\end{itemize}

A {\em Boolean algebra} is an ortholattice $L$ satisfying the distributive law
$$
x\wedge(y\vee z)=(x\wedge y)\vee(x\wedge z),
$$
for all $x,y,z\in L$.
An ortholattice is $L$ an {\em orthomodular lattice} if it satisfies the 
orthomodular law
$$
x\leq y\implies y=x\vee(y\wedge x^\perp),
$$
for all $x,y\in L$. 
It is easy to see that every Boolean algebra is an orthomodular lattice.

\subsection{Orthosets}

\begin{definition}
Let $V$ be a set, write $\pow(V)$ for the powerset of $V$. A mapping $\cl\colon
\pow(V)\to\pow(V)$ is a {\em closure operator} if it satisfies the conditions,
for all $X,Y\in\pow(V)$ 
\begin{enumerate}
\item $X\subseteq\cl(X)$,
\item $X\subseteq Y\implies \cl(X)\subseteq\cl(Y)$,
\item $\cl(\cl(X))=\cl(X)$.
\end{enumerate}
The pair $(V,\cl)$ is then called {\em a closure space}.
\end{definition}
In a closure space $(V,\cl)$, a set $X\subseteq V$ is called {\em closed} if
$\cl(X)=X$. It is easy to see that the set of all closed subsets of a closure space
forms a complete lattice under inclusion.

\begin{definition}
An {\em orthoset} $(O,\perp)$ is a set $O$ equipped with an irreflexive
symmetric binary relation $\perp\subseteq O\times O$.
\end{definition}

Let $(O,\perp)$ be an orthoset. 
We say that two elements $x,y\in O$ with $x\perp y$ are {\em
orthogonal}. A subset $B$ of $O$ is {\em pairwise orthogonal} if
every distinct pair of elements of $B$ is orthogonal.
Two subsets $X,Y$ of $O$ are {\em orthogonal} (in symbols $X\perp Y$) if 
$x\perp y$, for all $x\in X$ and $y\in Y$. We write $y\perp X$ for
$\{y\}\perp X$.
For every subset $X$ of $O$, we write 
\[
X^\perp=\{y\in O\mid y\perp X\}
\]
so that $X^\perp$ is the greatest subset of $O$ orthogonal to $X$. For singletons,
we abbreviate $\{x\}^\perp=x^\perp$.

Since the mapping $X\mapsto X^\perp$ is antitone, the mapping $X\mapsto X^\pp$ is
isotone.
Moreover, for every $X\subseteq O$ we have 
$X\subseteq X^\pp$ and $X^\perp=X^{\pp\perp}$. Consequently, $X\mapsto X^\pp$
is a closure operator.  Note that $s\in X^\pp$ if and only if, for all $a\in
O$,
$$
a\perp X\implies s\perp a.
$$
We say that a subset $X$ of $O$ is {\em orthoclosed} if $X=X^\pp$.
Note that for every $Y\subseteq O$, $Y^\perp$ is orthoclosed.

The set of all orthoclosed subsets of an orthoset forms a complete ortholattice $L(O,\perp)$,
with meets given by intersection and joins given by 
$$
X\vee Y=(X\cup Y)^\pp
$$
Note that $X\vee Y=(X^\perp\cap Y^\perp)^\perp$.
The smallest element of $L(O,\perp)$ is $\emptyset$ and the greatest element of
$L(O,\perp)$ is $O$. We say that $L(O,\perp)$ is {\em the logic} of $O$.

The following characterization of orthocomplements will be useful for our purposes.
\begin{lemma}\label{lemma:orthocomplements}
Let $(O,\perp)$ be an orthoset. Let $X,Y$ be subsets of $O$.
Then $X$ is orthoclosed and $Y=X^\perp$ if and only if $X\perp Y$ and
for every $z\in O\setminus(X\cup Y)$ there are $x\in X$, $y\in Y$ such that
$z\not\perp x$ and $z\not\perp y$.
\end{lemma}
\begin{proof}
$(\Rightarrow):$ 
Let $z\in O\setminus(X\cup Y)$. If for all $x\in X$ $z\perp x$,
then $z\in X^\perp=Y$, a contradiction. Hence $z\not\perp x$ for some $x\in X$.
Similarly, if for all $y\in Y$, $x\perp y$, then $z\in Y^\perp=X^\pp=X$ because
$X$ is orthoclosed. Therefore, $z\not\perp y$ for some $y\in Y$.

$(\Leftarrow):$  
Let us prove that $Y=X^\perp$. As $X\perp Y$, $Y\subseteq X^\perp$. For the proof
of the opposite inclusion, suppose that $X^\perp\setminus Y\neq\emptyset$ and let
$z\in X^\perp\setminus Y$.  Since $z\in X^\perp$, $z\notin X$ and since $z\notin
Y$, $z\in O\setminus(X\cup Y)$. By assumption, there are $x\in X$, $y\in Y$ such
that $z\not\perp x$ and $z\not\perp y$. However, $z\not\perp x\in X$ contradicts
$z\in X^\perp$, hence $X^\perp\setminus Y=\emptyset$ and $X^\perp\subseteq Y$.

By symmetry, $X=Y^\perp$. In particular, $X$ is orthoclosed.
\end{proof}

\subsection{Dacey orthosets}

Let $(O,\perp)$ be an orthoset, let $X$ be an orthoclosed subset of
$O$. A maximal pairwise orthogonal subset of $X$ is called a {\em basis of $X$}.

In his PhD. thesis, J.C. Dacey proved the following theorem.
\begin{theorem}\label{thm:dacey}\cite{dacey1968orthomodular}
Let $(O,\perp)$ be an orthoset. Then $L(O,\perp)$ is an orthomodular lattice if
and only if for every orthoclosed subset $X$ of $O$ and every basis $B$ of $X$, $X=B^\pp$. 
\end{theorem}

An orthoset satisfying the condition of \Cref{thm:dacey} is called a
{\em Dacey orthoset}. A orthoclosed subset $X$ in an orthoset such that
$X=B^\pp$, for every basis $B$ of $X$ is called a {\em Dacey subset}.

\begin{lemma}\cite{jenca2023orthogonality}
\label{lemma:chardacey}
Let $X$ be an orthoclosed subset
of an orthoset. The following are equivalent.
\begin{enumerate}[(a)]
\item $X$ is a Dacey subset.
\item For every basis $B$ of $X$, $B^\perp=X^\perp$. 
\item For every basis $B$ of $X$, $B^\perp\subseteq X^\perp$.
\end{enumerate}
\end{lemma}

\begin{example}
Let $P_4=\{1,2,3,4\}$, define the orthogonality relation on $P_4$ by the rule
$$
a\perp b\Leftrightarrow |a-b|=1
$$
(As a graph, $P_4$ is the 4-path.)
Let $X=\{1,3\}$. We see that $X^\pp=\{2\}^\perp=\{1,3\}=X$, so $X$ is orthoclosed.
The set $B=\{3\}$ is a basis of $X$. However, $B^\perp=\{3\}^\perp=\{2,4\}\not\subseteq\{2\}=X^\perp$ and, by \Cref{lemma:chardacey}, we see that $(P_4,\perp)$ is not Dacey.
\end{example}
\begin{example}
Let $\mathcal H$ be a Hilbert space.
Let $O=\mathcal H\setminus\{\vec 0\}$. Then $(O,\perp)$, where $\perp$ denotes
the orthogonality of vectors, is a Dacey orthoset. In this orthoset,
orthoclosed subsets correspond to closed subspaces of $\mathcal H$ and bases of
orthoclosed subsets correspond to orthogonal bases of those subspaces. The fact
that this is a Dacey space now simply follows from the fact that a maximal
orthogonal subset of a closed subspace must be an orthogonal basis of that subspace.
\end{example}

\subsection{N-free posets}

\begin{figure} \begin{center} \includegraphics{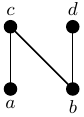} \end{center} \caption{The
N} \label{fig:TheN} \end{figure} Let $P$ be a poset.  For a quadruple of
elements $(a,b,c,d)\in P^4$, we say that they {\em form an N} if and only if
$a<c$, $b\prec c$ (note the covering relation here), $b<d$, and all the other distinct pairs of elements of the set $\{a,b,c,d\}$ are incomparable -- see \Cref{fig:TheN}. We denote this by $N(a,b,c,d)$. A poset
such that no quadruple of elements forms an N is called {\em N-free}. 

Note that
if $P$ is finite then $N(a,b,c,d)$ implies existence of $a'$ and $d'$ such that
$a'\prec c$, $b\prec d'$ and $N(a',b,c,d')$, that means, all the comparable
distinct pairs in $\{a',b,c,d'\}$ are covers. Therefore, in the finite case, a
poset is N-free iff it does not contain $(a',b,c,d')\in P^4$ such that
$N'(a',b,c,d')$.
This is sometimes used as the definition of N-free posets \cite{abdi2023proof}.

Nevertheless, there is another, inequivalent use of the term {\em N-free} -- the posets without quadruples $(a,b,c,d)$ with $a<d$, $b<c$, $b<d$ and all the other
pairs incomparable (so $c$ does not need to cover $b$). These posets are sometimes 
called {\em series-parallel posets} \cite{mohring1989computationally}.

N-free posets in our sense were considered first by Grillet in
\cite{grillet1969maximal} where (a more general version of) the following
theorem was proved.

\begin{theorem}\label{thm:chainantichain}
A finite poset $P$ is N-free if and only if every maximal chain in $P$ intersects
every maximal antichain in $P$.
\end{theorem}

In \cite{habib1985n}, it was shown that finite N-free posets can by constructed
from one-element posets by iterating the construction of disjoint union and a
certain generalization of the ordinal sum. This characterization lead to some
applications in the theory of concurrent processes in computer science, see for
example \cite{habib1991remarks,boudol1988concurrency}.

\section{Compatible orthosets}

\begin{definition}\label{def:compatible}
We say that an orthoset $(O,\perp)$ is \emph{compatible} if for all pairs $x,y\in O$
with $x\not\perp y$ there is $z\in O$ such that $x^\perp\cup y^\perp\subseteq z^\perp$.
\end{definition}

This definition is taken from the paper \cite{lu2010zero}. 
In that paper, the authors introduced a notion of a {\em compact graph} as
a graph-theoretic characterization on zero-divisor graphs of posets originally 
introduced and studied in \cite{halavs2009beck}.
A graph is compact if it is compatible in the sense of \Cref{def:compatible}
and contains no isolated vertices. 

\begin{theorem}\label{thm:compatible}
An orthoset is compatible if and only if $L(O,\perp)$ is a Boolean algebra.
\end{theorem}
\begin{proof}
Suppose that $(O,\perp)$ is compatible, that means, for all $x,y\in V(G)$ such that $x\not\perp y$, there is $z\in V(G)$ such that $x^\perp\cup y^\perp\subseteq z^\perp$. Note that
\begin{align*}
x^\perp\cup y^\perp\subseteq z^\perp&\Leftrightarrow \\
(x^\perp\cup y^\perp)^\perp\supseteq z^{\perp\perp}&\Leftrightarrow \\
x^{\perp\perp}\cap y^{\perp\perp}\supseteq z^{\perp\perp}&\Leftrightarrow \\
z\in x^{\perp\perp}\cap y^{\perp\perp}
\end{align*}
So we may reformulate the compatibility as
$$
x\not\perp y\implies x^{\perp\perp}\cap y^{\perp\perp}\neq\emptyset
$$
and this is equivalent to
\begin{equation}\label{eq:charcompatible}
x^{\perp\perp}\cap y^{\perp\perp}=\emptyset\implies x\perp y.
\end{equation}
Let us prove that \eqref{eq:charcompatible} is equivalent to $L(O,\perp)$ being Boolean. 

Suppose that $L(O,\perp)$ is Boolean. Recall \cite[Proposition 2.3]{jenca2023orthogonality}, that an ortholattice $L$ is a Boolean algebra if and only if it satisfies the property
\begin{equation}\label{eq:charbool}
a\wedge b=0\implies a\leq b^\perp,
\end{equation}
for all elements $a,b\in L$. Hence if $L(O,\perp)$ is Boolean, then $x^{\perp\perp}\cap y^{\perp\perp}=\emptyset$ implies that
$x^{\perp\perp}\subseteq y^{\perp\perp\perp}$ by \eqref{eq:charbool}, meaning that
$$
x\in x^{\perp\perp}\subseteq y^{\perp\perp\perp}=y^\perp,
$$
hence $x\in y^\perp$ and $x\perp y$.

Suppose that \eqref{eq:charcompatible} is satisfied and let $A,B$ be orthoclosed subsets with $A\cap B=\emptyset$. We need to prove that $A\subseteq B^\perp$, in other words,
that $x\in A$ and $y\in B$ imply that $x\perp y$. This is easy: since $A$ and $B$ are
orthoclosed, we obtain $x^{\perp\perp}\subseteq A$ and $y^{\perp\perp}\subseteq B$,
so $x^{\perp\perp}\cap y^{\perp\perp}=\emptyset$, because $A$ and $B$ are disjoint and hence, by \eqref{eq:charcompatible}, $x\perp y$.
\end{proof}
\begin{corollary}
Every compatible orthoset is Dacey.
\end{corollary}
\begin{proof}
Every Boolean algebra is an orthomodular lattice. The rest follows by \Cref{thm:dacey}.
\end{proof}

\section{Incomparability orthosets}

Let $P$ be a poset. For $x,y\in P$ we will write $x\perp y$ if and
only if $x,y$ are incomparable. We say that $(P,\perp)$ is the
{\em incomparability orthoset} of $P$.

Let us formulate a refinement of the \Cref{lemma:orthocomplements} for
incomparability orthosets.
\begin{lemma}
\label{lemma:updown}
Let $P$ be a poset, let $(P,\perp)$ be the incomparability
orthoset of $P$. Let $X,Y$ be subsets of $P$. Then
$X$ is orthoclosed and $Y=X^\perp$ if and only if $X\perp Y$ and for every $z\in P\setminus(X\cup Y)$
there are $x\in X$, $y\in Y$ such that either $z<x$, $z<y$ or $z>x$, $z>y$.
\end{lemma}
\begin{proof}
Suppose that $X$ is orthoclosed and that $Y=X^\perp$. Let $z\in P\setminus(X\cup Y)$.  By
\Cref{lemma:orthocomplements}, there are $x\in X$ and $y\in Y$ such that
$z\not\perp x$ and $z\not\perp y$. If $z\not\perp x$, then $z=x$ or $z<x$ or
$z>x$. However $z=x$ contradicts $z\notin X$ so $z\neq x$.  Similarly, $z\neq y$,
so either $z<y$ or $z>y$. Suppose that $z<x$. If $z>y$, then $y<z<x$, but this contradicts $X\perp Y$. Therefore, $z<y$. Similarly, $z>x$ implies $z>y$.

The opposite implication follows directly by \Cref{lemma:orthocomplements}.
\end{proof}
By \Cref{lemma:updown}, for every orthoclosed set $X$ in the incomparability
orthoset $(P,\perp)$ we have 
$$
P\setminus(X\cup X^\perp)=U_X\cup D_X
$$
where
\begin{align*}
D_X&=\{z\in P\setminus(X\cup X^\perp)\mid \exists x\in X,y\in X^\perp: z<x,z<y\} \\
U_X&=\{z\in P\setminus(X\cup X^\perp)\mid \exists x\in X,y\in X^\perp: x<z,y<z\}.
\end{align*}

Note that $U_X\cap D_X=\emptyset$. Indeed, if $z\in U_X\cap D_X$ then $x<z<y$, where $x\in X$ and $y\in X^\perp$, contradicting $X\perp X^\perp$.

\subsection{The orthomodular case}

\begin{theorem}\label{thm:main}
Let $P$ be a finite poset. The following are equivalent.
\begin{enumerate}[(a)]
\item $P$ is N-free.
\item $(P,\perp)$ is Dacey.
\item $L(P,\perp)$ is an orthomodular lattice.
\end{enumerate}
\end{theorem}
\begin{proof}
The equivalence of (b) and (c) follows by \Cref{thm:dacey}.

(b)$\implies$(a):
Assume that $N(a,b,c,d)$ in $P$.
We will prove that this implies
existence of a non-Dacey subset of $(P,\perp)$.
Let
$$
X=\{x\in P:x<c\text{ and }x\nleq d\}.
$$
As $a\in X$, $X$ is nonempty. We will prove that the orthoclosed set $X^\perp$ is
not Dacey. We need to prove three auxiliary claims first.

\begin{claim}
$d\in X^\perp$. 
\end{claim}
\noindent{\em Proof of the claim.}
If $d\not\perp X$, then there is a $x\in X$ such that $d\not\perp x$,
so either $x\leq d$ or $d<x$. However, $x\leq d$ contradicts $x\in X$ and 
$d<x$ implies $d<x<c$, but $d\nless c$.

Let $B$ be a basis of the orthoclosed set $X^\perp$ with $d\in B$.

\begin{claim}
$c\in B^\perp$.
\end{claim}
\noindent{\em Proof of the claim.}
Assume that $c\not\perp B$. Then there is $y\in B$ such that
$c\not\perp y$. Note that $c\perp d$, so $y\neq d$ and since $y,d$ are distinct
elements of the orthogonal set $B$, $y\perp d$. As $c\not\perp y$, either
$c\leq y$ or $y<c$. If $c\leq y$, then $a<c\leq y$, which implies $y\not\perp a\in X$,
contradicting $y\in X^\perp$. Assume that $y<c$. We already know that $y\perp d$,
hence $y\nleq d$. However, $y<c$ and $y\nleq d$ mean that $y\in X$, contradicting
$y\in X^\perp$.

\begin{claim}
$b\in X^\perp$
\end{claim}
\noindent{\em Proof of the claim.}
Assume that $b\not\perp X$. There is $x\in X$ such that $b\not\perp x$.
If $b<x$, then $b<x<c$ and this contradicts $b\prec c$. If $x\leq b$, then
$x\leq b<d$, contradicting $x\in X$.

As $b\in X^\perp$ and $b\leq c$, $c\notin X^{\perp\perp}$, hence $c\in B^\perp$ and
$c\notin X^{\perp\perp}$, so $B^\perp \supset X^{\perp\perp}$ and $X^\perp$ is not Dacey.

(a)$\implies$(b): Assume that $(P,\perp)$ is not Dacey.  We shall prove that
this implies that $P$ contains a covering N.

There is some non-Dacey orthoclosed subset $X$ of $(P,\perp)$. As $X$ is not Dacey,
there is a basis $B$ of $X$ with $X^\perp\subset B^\perp$. Let $c\in B^\perp\setminus
X^\perp$.

As $c\in B^\perp$, $c\perp B$, so $c\notin X$ because $B$ is a maximal orthogonal subset
of $X$. Therefore, $c\in P\setminus(X\cup X^\perp)$. By \Cref{lemma:updown}, $c\in U_X$ or $c\in
D_X$. We shall assume that $c\in U_X$; the other case follows by duality.

We see that $c\in U_X\cap(B^\perp\setminus X^\perp)$. However, since $U_X\cap
X^\perp=\emptyset$, $U_X\cap(B^\perp\setminus X^\perp)=U_X\cap B$. Let us
assume that $c$ is a minimal element of the set $U_X\cap B^\perp$.

Since $c\in U_X$, there are $x\in X,a\in X^\perp$ such that $x<c$ and $a<c$. As $x<c$, the
semiopen interval
$$
[x,c)=\{w\in P:w<c\text{ and }x\leq w\}
$$
is nonempty. Let $b$ be a maximal element of $[x,c)$; note that $b\prec c$.

Suppose that $b\notin X$. As $x\leq b$, $b\notin X^\perp$, so $b\in
P\setminus(X\cup X^\perp)$ and and $X\ni x<b$. By \Cref{lemma:updown}, there is $y\in
X^\perp$ such that $y<b$.  As $X\ni x<b>y\in X^\perp$, $b\in U_X$. 

If $b\in B^\perp$, then $b\in U_X\cap B^\perp$, contradicting the
minimality of $c$ in $U_X\cap B^\perp$. So $b\notin B^\perp$, meaning that
there is $w\in B$ such that $b\not\perp w$. Note that $w\leq b<c$ contradicts
$c\in B^\perp$, hence $b<w$ and we see that $N(a,b,c,w)$.

Suppose that $b\in X$. Then $b\notin B$, because $c\perp B$ and $c\not\perp b$.  As $B$ is a
basis of $X$, $b\in X$ and $b\notin B$ imply that $b\not\perp B$. So there is a $d\in B$
such that $b\not\perp d$. If $d\leq b$, then $B\ni d\leq b<c$, contradicting $c\perp B$.
Hence $b<d$ and we see that $N(a,b,c,d)$.
\end{proof}

\subsection{The Boolean case}

Let $P$ be a poset. 
We say that $(a,b,c,d)\in P^4$ form \emph{a weak N}, if
$a<c$, $b\prec c$, $b<d$, $a\perp b$,  $c\perp d$. We denote this by
$N^?(a,b,c,d)$. Note that $N(a,b,c,d)$ iff $N^?(a,b,c,d)$ and $a\perp d$ or,
in other words, in a weak N the incomparability of $a$ and $d$ is not required.
In particular, $N(a,b,c,d)$ implies $N^?(a,b,c,d)$, so a poset without a
weak N is certainly N-free.

\begin{theorem}
For every finite poset $P$, the following are equivalent.
\begin{enumerate}[(a)]
\item There is no weak N in $P$.
\item $(P,\perp)$ is a compatible orthoset.
\item $L(P,\perp)$ is Boolean. 
\end{enumerate}
\end{theorem}
\begin{proof}
By \Cref{thm:compatible}, (b) is equivalent to (c).

(a)$\implies$(b): Suppose that we have $N^?(a,b,c,d)$ in $P$. As $b\prec c$, $c\not\perp b$. We will prove that there is
no $z\in P$ such that $c^\perp\cup b^\perp\subseteq z^\perp$. 

Assume the contrary; this implies that $z\not\perp b$ and $z\not\perp c$.
Indeed, $z\perp b$ would mean that $z\in b^\perp\subseteq z^\perp$ and hence $z\perp z$ 
which is not true, so $z\not\perp b$. Similarly $z\not\perp c$. Furthermore, as $a\in b^\perp$
and $d\in c^\perp$, we see that $\{a,d\}\subseteq b^\perp\cup c^\perp\subseteq z^\perp$, so
$z\perp a$ and $z\perp d$.

We have proved that $z\not\perp b$. This is equivalent to $z\leq b$ or $b<z$. However,
$z\leq b<d$ contradicts $z\perp d$, so $b<z$. Similarly, we can prove that $z<c$. But $b<z<c$
contradicts with our assumption that $b\prec c$.

\noindent(b)$\implies$(a): Suppose that $(P,\perp)$ is not compatible. We will
prove that there is some weak N in $P$.

Since $(P,\perp)$ is not compatible, there are $x,y\in P$ with $x\not\perp y$ 
such that there does not exist $z\in P$ such that $x^\perp\cup y^\perp\subseteq z^\perp$.
Equivalently, there are $x,y\in P$ with $x\not\perp y$, such that for all $z\in P$, $x^\perp\cup y^\perp\not\subseteq z^\perp$.

Clearly, $x\neq y$, because otherwise we could put $z=x$. Hence $x\not\perp y$
implies that either $x<y$ or $x>y$. Without loss of generality, we may assume
that $x<y$.  If $x^\perp\subseteq y^\perp$, then $x^\perp\cup y^\perp=y^\perp$
and we may put $z=y$, contradicting our assumptions. So $x^\perp\not\subseteq
y^\perp$ and, similarly, $y^\perp\not\subseteq x^\perp$. Hence there must be
$a\in x^\perp\setminus y^\perp$ and $d\in y^\perp\setminus x^\perp$.
Summarizing, we see that $a\perp x$, $a\not\perp y$, $d\perp y$ and $d\not\perp x$.

As $a\not\perp y$, either $a<y$ or $y\leq a$. If $y\leq a$, then $x<y\leq a$,
but this contradicts $a\perp x$. So $a<y$ and, similarly, we can prove that
$x<d$. If $x\prec y$, then $N^?(a,x,y,d)$ and we are done. If $x\not\prec y$,
then there is $z\in P$ be such that $x<z<y$. For any such $z$, we cannot have
$x^\perp\cup y^\perp\subseteq z^\perp$, that means, there is $w\in x^\perp\cup
y^\perp$ such that $w\not\perp z$. Since $z\not\in x^\perp\cup y^\perp$, $w\neq
z$. Hence, for every $z\in P$ with $x<z<y$ there is $w\in x^\perp\cup y^\perp$ such that
either $w<z$ or $z<w$. However, if $w<z$, then $w<z<y$ implies $w\notin y^\perp$, so we must
have $w\in x^\perp$. Similarly $z<w$ implies $w\in y^\perp$.

Having this property of every such $z$ in mind, consider a maximal chain 
$$
x\prec z_n\prec\dots\prec z_1\prec y
$$
in the closed interval $[x,y]$ of $P$. There is a sequence $w_1,\dots,w_n$
such that for each $i$ either ($w_i<z_i$ and $w_i\perp x$) or ($z_i<w_i$ and $w_i\perp y$). 

Suppose that $z_1<w_1$. Then $N^?(a,z_1,y,w_1)$ and we are done, so suppose that
$w_1<z_1$. Let $k$ be the greatest $k\in\{1,\dots,n\}$ such that 
such that $w_k<z_k$. If $k=n$, then $N^?(w_n,x,z_n,d)$ and we are done.
If $k<n$ then $z_{k+1}<w_{k+1}$ by maximality of $k$ and then
$N^?(w_k,z_{k+1},z_k,w_{k+1})$.
\end{proof}

\section{Conclusion}

Since the characterization of N-free posets in \Cref{thm:chainantichain} is chain--antichain symmetric,
at this point a natural question arises: can we replace the incomparability
relation in \Cref{thm:main} by {\em strict comparability}? The following
example shows that this is not possible.  
\begin{example} Consider the poset $P$ on
\begin{figure}
\begin{center}
\includegraphics{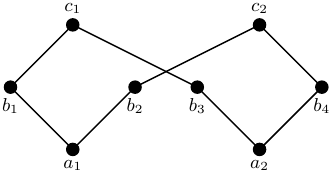}
\end{center}
\caption{The counterexample $P$}
\label{fig:counterex}
\end{figure}
\Cref{fig:counterex}. It is easy to see that it is N-free. The {\em strict
comparability orthospace} is $(P,\lessgtr)$, where $x\lessgtr y$ means that
$x<y$ or $x>y$. 
The {\em strict comparability orthospace} of $P$ is $(P,\lessgtr)$, where $x\lessgtr y$ means that $x<y$ or $x>y$.

Consider the sets $X=\{a_1,a_2\}$ and $Y=\{c_1,c_2\}$. We see that
$X=Y^{\lessgtr}$ and $Y=X^{\lessgtr}$, so $X$ is orthoclosed in $(P,\lessgtr)$. Pick a basis
$B=\{a_1\}$ of $X$. Clearly, $b_1\in B^{\lessgtr}$ and $b_1\notin X^\lessgtr=Y$, hence $X$ is not Dacey.
\end{example}

Therefore, we can not replace ``incomparability'' by ``strict comparability''
in \Cref{thm:main}. Let us close with an open problem.
\begin{problem}
Characterize finite posets for which the strict comparability orthospace is Dacey.
\end{problem}

\section*{Declarations}
\begin{description}
\item[\bf Funding:]
This research is supported by grants VEGA 2/0128/24 and 1/0036/23,
Slovakia and by the Slovak Research and Development Agency under the contracts APVV-20-0069
and APVV-23-0093.
\end{description}

\end{document}